\documentclass[reqno]{amsart}
\usepackage{amsmath, amsthm, amssymb, amstext}
\usepackage{hyperref,xcolor}
\definecolor{bluecite}{HTML}{0875b7}
\hypersetup{
pdfborder={0 0 0},
colorlinks,
linkcolor=bluecite,
citecolor=bluecite
}

\usepackage{enumitem}
\allowdisplaybreaks

\newtheorem{theorem}{Theorem}

\newtheorem{proposition}[theorem]{Proposition}
\newtheorem{corollary}[theorem]{Corollary}
\newtheorem{definition}[theorem]{Definition}

\DeclareMathOperator*{\divergenz}{div}              %

\newcommand*\diff{\mathop{}\!\mathrm{d}}

\newcommand{\N}{\mathbb{N}}

\newcommand{\R}{\mathbb{R}}

\newcommand{\W}{W^{1,\mathcal{H}}(\Omega)}
\newcommand{\Lp}[1]{L^{#1}(\Omega)}

\newcommand{\Lprand}[1]{L^{#1}(\partial\Omega)}
\newcommand{\Wp}[1]{W^{1,#1}(\Omega)}

\newcommand{\Wpzero}[1]{W^{1,#1}_0(\Omega)}

\newcommand{\eps}{\varepsilon}
\newcommand{\ph}{\varphi}

\newcommand{\into}{\int_{\Omega}}

\newcommand{\weak}{\rightharpoonup}

\newcommand{\Linf}{L^{\infty}(\Omega)}

\renewcommand{\l}{\left}
\renewcommand{\r}{\right}
\numberwithin{theorem}{section}
\numberwithin{equation}{section}

\allowdisplaybreaks

\title[Singular Finsler double phase problems ]{Singular Finsler double phase problems with nonlinear boundary condition}

\author[C.\,Farkas]{Csaba Farkas}
\address[C.\,Farkas]{Sapientia Hungarian University of Transylvania, Department of Mathematics and Computer Science, T\^{a}rgu Mure\textcommabelow{s}, Romania}
\email{farkascs@ms.sapientia.ro, farkas.csaba2008@gmail.com}

\author[A.\,Fiscella]{Alessio Fiscella}
\address[A.\,Fiscella]{Departamento de Matem\'{a}tica, Universidade Estadual de Campinas, IMECC, Rua S\'{e}rgio Buarque de Holanda, 651, Campinas, SP CEP 13083-859, Brazil}
\email{fiscella@ime.unicamp.br}

\author[P.\,Winkert]{Patrick Winkert}
\address[P.\,Winkert]{Technische Universit\"{a}t Berlin, Institut f\"{u}r Mathematik, Stra\ss e des 17.\,Juni 136, 10623 Berlin, Germany}
\email{winkert@math.tu-berlin.de}

\subjclass{35J15, 35J62, 58B20, 58J60}
\keywords{Anisotropic double phase operator, critical type exponent, existence results, Minkowski space, nonlinear boundary condition, singular problems}

\begin{document}

\begin{abstract} 
	In this paper we study a singular Finsler double phase problem with a nonlinear boundary condition and perturbations that have a type of critical growth, even on the boundary. Based  on variational methods in combination with truncation techniques we prove the existence of at least one weak solution for this problem under very general assumptions. Even in the case when the Finsler manifold reduces to the Euclidean norm, our work is the first one dealing with a singular double phase problem and nonlinear boundary condition.
\end{abstract}
\maketitle

\section{Introduction}
In this paper we consider singular Finsler double phase problems with nonlinear boundary condition. The Finsler double phase operator is defined by
\begin{align}\label{finsler_double_phase_operator}
	\divergenz (A(u)):=\divergenz\big(F^{p-1}(\nabla u)\nabla F(\nabla u) +\mu(x) F^{q-1}(\nabla u)\nabla F(\nabla u)\big)
\end{align}
for $u\in \W$ with $\W$ being the Musielak-Orlicz Sobolev space  and $F$ is a positive homogeneous function such that $F\in C^\infty(\mathbb R^N\setminus \{0\})$ and the Hessian matrix $\nabla^2 ({F^2}/{2})(x)$ is positive definite for all $x\neq 0$. Furthermore, $\mu$ is a nonnegative bounded function and $1<p<q<N$. If $F$ coincides with the Euclidean norm, that is, $F(\xi)=\left(\sum_{i=1}^N |\xi_i|^2\right)^{1/2}$ for $\xi\in\R^N$, then \eqref{finsler_double_phase_operator} reduces to the usual double phase operator given by
\begin{align}\label{double_phase_operator}
	\divergenz \big(|\nabla u|^{p-2} \nabla u+ \mu(x) |\nabla u|^{q-2} \nabla u\big).
\end{align}
Also, if $\mu\equiv 0$ or $\inf_{\overline{\Omega}}\mu>0$, then \eqref{double_phase_operator} (similarly \eqref{finsler_double_phase_operator}) reduces to the (Finsler) $p$-Laplacian or the (Finsler) $(q,p)$-Laplacian, respectively. The study of such operators and corresponding energy functionals are motivated by physical phenomena, see, for example, the work of Zhikov \cite{Zhikov-1986} (see also the monograph of Zhikov-Kozlov-Ole\u{\i}nik \cite{Zhikov-Kozlov-Oleinik-1994}) in order to describe models for strongly anisotropic materials. Related functionals to \eqref{double_phase_operator} have been studied intensively with respect to regularity properties of local minimizers, see the works of Baroni-Colombo-Mingione \cite{Baroni-Colombo-Mingione-2015,Baroni-Colombo-Mingione-2016,Baroni-Colombo-Mingione-2018}, Baroni-Kuusi-Mingione \cite{Baroni-Kuusi-Mingione-2015}, Byun-Oh \cite{Byun-Oh-2020}, Colombo-Mingione \cite{Colombo-Mingione-2015a,Colombo-Mingione-2015b}, De Filippis-Palatucci \cite{De-Filippis-Palatucci-2019},
Marcellini \cite{Marcellini-1991,Marcellini-1989b}, Ok \cite{Ok-2018,Ok-2020}, Ragusa-Tachikawa \cite{Ragusa-Tachikawa-2020} and the references therein.

On the other hand, the minimization of the functional $E_F\colon H^1_0(\Omega)\to \R$ defined by
\begin{align*}
	E_F(u)= \int_\Omega F^2(\nabla u)\diff x \quad \text{for }u\in H^1_0(\Omega),
\end{align*}
under certain constraints on perimeter or volume occurs in many subjects of mathematical physics. Here the minimizer corresponds to an optimal shape (or configuration) of anisotropic tension-surface. The minimization of the functional $E_F$  describes, for example, the specific polyhedral shape of crystal structures in solid crystals with sufficiently small grains, as shown by Dinghas \cite{Dinghas-1944} and Taylor \cite{Taylor-2002}. It is clear that $E_F$ is the energy functional to the Finsler Laplacian given by 
\begin{align}\label{finsler_laplace_operator}
	\Delta_F u= {\rm div} (F(\nabla u)\nabla F(\nabla u)).
\end{align}
The Finsler Laplacian, given in \eqref{finsler_laplace_operator}, has been studied by several authors in the last decade. We refer, for example, to the papers of Cianchi-Salani \cite{Cianchi-Salani}  and Wang-Xia \cite{Wang-Xia-2011}, both dealing with the Serrin-type overdetermined anisotropic problem, or to Farkas-Fodor-Krist\'aly \cite{Farkas-Fodor-Kristaly-2015} who studied a sublinear Dirichlet problem of this type.  Related works concerning anisotropic phenomena can be found in the works of Bellettini-Paolini \cite{Bellettini-Paolini-1996}, Belloni-Ferone-Kawohl \cite{Belloni-Ferone-Kawohl-2003}, Della Pietra-Gavitone \cite{DellaPietra-Gavitone-2016}, Della Pietra-di Blasio-Gavitone \cite{DellaPietra-diBlasio-Gavitone-2020}, Della Pietra-Gavitone-Piscitelli \cite{DellaPietra-Gavitone-Piscitelli-2019}, Farkas \cite{Farkas-2020}, Farkas-Krist\'aly-Varga \cite{Farkas-Kristaly-Varga-2015}, Ferone-Kawohl \cite{Ferone-Kawohl-2009} and the references therein.

In this paper we combine the effect of a Finsler manifold and a double phase operator along with a singular term and a nonlinear boundary condition. More precisely, we study the following problem
\begin{equation}\label{problem}
	\begin{aligned}
		-\divergenz (A(u)) +u^{p-1}+\mu(x)u^{q-1}& = u^{p^{*}-1}+\lambda \l(u^{\gamma-1}+ g_1(x,u)\r) \quad && \text{in } \Omega,\\
		A(u)\cdot \nu & = u^{p_*-1}+g_2(x,u) &&\text{on } \partial \Omega,\\
		u & > 0 &&\text{in } \Omega,\\
	\end{aligned}
\end{equation}
where $\Omega \subset \R^N$, $N\geq 2$,  is a bounded domain with Lipschitz boundary $\partial\Omega$, $\nu(x)$ is the outer unit normal of $\Omega$ at the point $x \in \partial\Omega$, $\lambda$ is a positive parameter and  the following assumptions hold true:
\begin{enumerate}
	\item[\textnormal{(H)}]
	\begin{enumerate}
		\item[\textnormal{(i)}]
			$0<\gamma<1$ and
			\begin{align}\label{conditions-exponents}
				1 < p<q<N,\quad  q<p^*,\quad 0 \leq \mu(\cdot)\in \Linf;
			\end{align}
		\item[\textnormal{(ii)}]
			$g_1\colon\Omega\times\R\to\R$ and $g_2\colon \partial\Omega \times \R\to\R$ are Carath\'eodory functions and there exist $1<\theta_1<p\leq\nu_1<p^*$, $p< \nu_2<p_*$ as well as nonnegative constants $a_1, a_2, b_1$ such that
			\begin{equation*}
				\begin{aligned}
					\qquad \qquad\qquad g_1(x,s)&\leq a_{1}s^{\nu_1-1}+b_{1}s^{\theta_1-1}&&\text{for a.\,a.\,}x\in\Omega \text{ and for all } s\geq 0,\\
					g_2(x,s)&\leq a_{2}s^{\nu_2-1} &&\text{for a.\,a.\,}x\in\partial \Omega \text{ and for all } s\geq 0,
				\end{aligned}
			\end{equation*}
			where $p^*$ and $p_*$  are the critical exponents to $p$ given by
			\begin{align}\label{critical_exponent}
				p^*:=\frac{Np}{N-p}\quad\text{and}\quad p_*:=\frac{(N-1)p}{N-p};
			\end{align}
		\item[\textnormal{(iii)}]
			the function $F\colon\R^N \to [0,\infty)$ is a positively homogeneous Minkowski norm with finite reversibility $\displaystyle r_F=\max_{w\neq 0}\frac{F(-w)}{F(w)}$.
	\end{enumerate}
\end{enumerate}

Since we are looking for positive solutions and hypothesis \textnormal{(H)(ii)} concerns the positive semiaxis $\R_+=[0,\infty)$, without any loss of generality, we may assume that $g_1(x,s)=g_2(x,s)=0$ for all $s\leq 0$ and for a.\,a.\,$x\in\Omega$, respectively, $x\in\partial\Omega$. Moreover, note that we always have $r_F \geq 1$, see for example in Farkas-Krist\'{a}ly-Varga \cite{Farkas-Kristaly-Varga-2015}. It is clear that the Euclidean norm has finite reversibility. Finally, we observe that \eqref{conditions-exponents} implies that $\W \hookrightarrow \Lp{q}$ compactly, as shown in Section \ref{section_2}.

\begin{definition}\label{def_weak_solution}
A function $u\in \W$ is called a weak solution of problem \eqref{problem} if $u^{\gamma-1}\varphi\in L^{1}(\Omega)$, $u>0$ for a.\,a.\,$x\in\Omega$
and if
\begin{align*}
	\begin{split}
		& \int_{\Omega}\left(F^{p-1}(\nabla u)\nabla F(\nabla u) +\mu(x) F^{q-1}(\nabla u)\nabla F(\nabla u)\right)\cdot \nabla\varphi\diff x\\ 
		&+\into u^{p-1}\ph \diff x+\into \mu(x) u^{q-1}\ph\diff x\\
		&=\int_{\Omega}u^{p^{*}-1}\varphi\diff x+\lambda \into \l( u^{\gamma-1}+g_1(x,u)\r)\varphi\diff x
		+\int_{\partial \Omega}\l(u^{p_*-1}+g_2(x,u)\r) \varphi\diff \sigma
	\end{split}
\end{align*}
is satisfied for all $\varphi\in\W$.
\end{definition}
From hypotheses \textnormal{(H)} we know that the definition of a weak solution is well defined.

The main result in this paper is the following theorem.

\begin{theorem}\label{main_theorem}
	Let hypotheses \textnormal{(H)} be satisfied. Then there exists $\lambda_*>0$ such that for every $\lambda\in (0,\lambda_*)$, problem \eqref{problem} has a nontrivial weak solution. 
\end{theorem}

To the best of our knowledge, this is the first work on a singular double phase problem with nonlinear boundary condition even in the Euclidean case, that is when $F(\xi)=\left(\sum_{i=1}^N |\xi_i|^2\right)^{1/2}$ for $\xi\in\R^N$. The novelty of our paper is not only due by the combination of the Finsler double phase operator with a singular term and nonlinear boundary condition. Indeed, in \eqref{problem} we also deal with a type of critical Sobolev nonlinearities, even on the boundary, related to the lower exponent $p$, as explained in \eqref{critical_exponent}. Such critical terms make the study of compactness of the energy functional related to \eqref{problem} more intriguing, since the embeddings $\W\hookrightarrow L^{p^*}(\Omega)$ and $\W\hookrightarrow L^{p_*}(\partial\Omega)$ are not compact. We overcome these difficulties with a local analysis on a suitable closed convex subset of $\W$ combined with a truncation argument.

We point out that $p^*$ and $p_*$ are not the critical exponents to the space $\W$. Indeed, from Fan \cite{Fan-2012} we know that $\W \hookrightarrow L^{\mathcal{H}_*}(\Omega)$ is continuous while $\mathcal{H}_*$ is the Sobolev conjugate function of $\mathcal{H}$, see also Crespo-Blanco-Gasi\'nski-Harjulehto-Winkert \cite[Definition 2.18 and Proposition 2.18]{Crespo-Blanco-Gasinski-Harjulehto-Winkert-2021}. So far it is not known how $\mathcal{H}_*$ explicitly looks like in the double phase setting. For the moment, $p^*$ and $p_*$ seem to be the best exponents (probably not optimal) and only continuous (in general noncompact) embeddings from $\W \hookrightarrow \Lp{p^*}$ and $\W\hookrightarrow\Lprand{p_*}$ are available. So we call it ``types of critical growth''.

For singular double phase problems with Dirichlet boundary condition there exists only few works. Recently, Liu-Dai-Papageorgiou-Winkert \cite{Liu-Dai-Papageorgiou-Winkert-2021} studied the singular problem
\begin{equation}\label{problem10}
	\begin{aligned}
		-\divergenz\big(|\nabla u|^{p-2} \nabla u+ \mu(x) |\nabla u|^{q-2} \nabla u\big)&= a(x)u^{-\gamma}+\lambda u^{r-1}\quad && \text{in } \Omega, \\
		u&= 0 && \text{on } \partial\Omega.
	\end{aligned}
\end{equation}
Based on the fibering method along with the Nehari manifold, the existence of at least two weak solutions with different energy sign is shown, see also \cite{Crespo-Blanco-Papageorgiou-Winkert-2021} for the corresponding Neumann problem. Furthermore, under a different treatment, Chen-Ge-Wen-Cao \cite{Chen-Ge-Wen-Cao-2020} considered problems of type \eqref{problem10} and proved the existence of a weak solution having negative energy. Finally, the existence of at least one weak solution to the singular problem 
\begin{equation*}
	\begin{aligned}
		-\divergenz (A(u)) & = u^{p^{*}-1}+\lambda \l(u^{\gamma-1}+g(u)\r)\quad && \text{in } \Omega,\\
		u & > 0 &&\text{in } \Omega,\\
		u & = 0 &&\text{on } \partial \Omega,
	\end{aligned}
\end{equation*}
has been shown by the first and the third author in \cite{Farkas-Winkert-2021}. The current paper can be seen as a nontrivial extension of the one in \cite{Farkas-Winkert-2021} to the case of a nonlinear boundary condition including type of critical growth. In particular, we are able to cover the situation when $1<p<2$ and/or $1<q<2$, which has not been considered in \cite{Farkas-Winkert-2021} where $2\leq p<q$.

Also for the $p$-Laplacian or the $(q,p)$-Laplacian only few works exist involving singular terms and Neumann/Robin boundary condition. We refer to Papageorgiou-R\u{a}dulescu-Repov\v{s} \cite{Papageorgiou-Radulescu-Repovs-2020b,Papageorgiou-Radulescu-Repovs-2021} for singular homogeneous Neumann $p$-Laplace problems and for singular Robin $(q,p)$-Laplacian problems, respectively. Existences results for singular Neumann Laplace problems has been obtained by Lei \cite{Lei-2018} based on variational and perturbation methods.

Finally, existence results for double phase problems without singular term can be found in the papers of Colasuonno-Squassina \cite{Colasuonno-Squassina-2016}, El-Manouni-Marino-Winkert \cite{El-Manouni-Marino-Winkert-2020}, Fiscella \cite{Fiscella-2020}, Fiscella-Pinamonti \cite{Fiscella-Pinamonti-2020},  Gasi\'nski-Papa\-georgiou \cite{Gasinski-Papageorgiou-2019}, Gasi\'nski-Winkert \cite{Gasinski-Winkert-2020a,Gasinski-Winkert-2020b,Gasinski-Winkert-2021}, Liu-Dai \cite{Liu-Dai-2018}, Papageorgiou-R\u{a}dulescu-Repov\v{s} \cite{Papageorgiou-Radulescu-Repovs-2020a}, Perera-Squassina \cite{Perera-Squassina-2019}, Zeng-Bai-Gasi\'nski-Winkert \cite{Zeng-Bai-Gasinski-Winkert-2020, Zeng-Gasinski-Winkert-Bai-2020} and the references there\-in. For related works dealing with certain types of double phase problems we refer to the works of Bahrouni-R\u{a}dulescu-Winkert \cite{Bahrouni-Radulescu-Winkert-2020}, Barletta-Tornatore \cite{Barletta-Tornatore-2021}, Faraci-Farkas \cite{Faraci-Farkas-2015}, Papageorgiou-R\u{a}dulescu-Repov\v{s} \cite{Papageorgiou-Radulescu-Repovs-2019b}, Papageorgiou-Winkert \cite{Papageorgiou-Winkert-2019} and Zeng-Bai-Gasi\'nski-Winkert \cite{Zeng-Bai-Gasinski-Winkert-2020b}.
\section{Preliminaries}\label{section_2}

In this section we are going to mention the main facts about the Minkowski space $(\R^N,F)$ and the properties about Musielak-Orlicz Sobolev spaces.

To this end, let $F\colon \R^N\to [0,\infty)$ be a positively homogeneous Minkowski norm, that is, $F$ is a positive homogeneous function such that $F\in C^\infty(\mathbb R^N\setminus \{0\})$ and the Hessian matrix $\nabla^2 ({F^2}/{2})(x)$ is positive definite for all $x\neq 0$. We point out that the pair $(\mathbb R^N,F)$ is the simplest not necessarily reversible Finsler manifold whose flag curvature is identically zero, the geodesics are straight lines and the intrinsic distance between two points $x,y\in \mathbb R^N$ is given by
\begin{equation*}
	d_F(x,y)=F(y-x).
\end{equation*}
The pair $(\R^N,d_F)$ is a quasi-metric space and in general, it holds $d_F(x,y)\neq d_F(y,x)$. 

The so-called Randers metric is a typical example for a Minkowski norm with finite reversibility which is given by 
\begin{align*}
	F(x)=\sqrt{\langle Ax ,x \rangle}+\langle b,x \rangle,
\end{align*}
where $A$ is a  positive definite and symmetric  $(N\times N)$-type matrix and  $b=(b_i)\in \mathbb R^N$ is a fixed vector such that $\sqrt{\langle A^{-1}b,b\rangle}<1$. Note that
\begin{align*}
	r_F=\frac{1+\sqrt{\langle A^{-1}b,b\rangle}}{1-\sqrt{\langle A^{-1}b,b\rangle}}.
\end{align*}
The pair $(\mathbb R^N,F)$ is often called Randers space which describes the electromagnetic field of the physical space-time in general relativity, see Randers \cite{Randers-1941}. They are deduced as the solution of the Zermelo navigation problem.

In the next proposition we recall some basic properties of $F$, see Bao-Chern-Shen \cite[\S 1.2]{Bao-Chern-Shen-2000}.

\begin{proposition}\label{basic-properties}
	Let $F\colon\R^N\to [0,\infty)$ be a positively homogeneous Minkowski norm. Then, the following assertions hold true:
	\begin{enumerate}
		\item[\textnormal{(i)}] 
		Positivity: $F(x)>0$ for all $x\neq 0$;
		\item[\textnormal{(ii)}]
		Convexity: $F$ and $F^2$ are strictly convex;
		\item[\textnormal{(iii)}] 
		Euler's theorem: $x\cdot \nabla F(x)=F(x)$ and 
		\begin{align*}
			\nabla^2 ({F^2}/{2})({x}){x}\cdot {x}=F^2(x)\quad  \text{for all } x\in\R^N\setminus \{0\};
		\end{align*}
		\item[\textnormal{(iv)}]
		Homogeneity: $\nabla F(tx)=\nabla F(x)$ and 
		\begin{align*}
			\nabla^2 F^2(tx)=\nabla^2 F^2(x)\quad \text{for all }x\in \R^N\setminus \{0\} \text{ and for all } t> 0.
		\end{align*}
	\end{enumerate}
\end{proposition}

Furthermore, $\Lp{r}$ and $L^r(\Omega;\R^N)$ stand for the usual Lebesgue spaces endowed with the norm $\|\cdot\|_r$ for $1\leq r<\infty$. The corresponding Sobolev spaces are denoted by $\Wp{r}$ and $\Wpzero{r}$ equipped with the norms
\begin{align*}
	\|u\|_{1,r,F}= \|F(\nabla u) \|_{r}+\|u\|_{r} \quad\text{and}\quad 
	\|u\|_{1,r,0,F}= \|F(\nabla u) \|_{r},
\end{align*}
respectively. 

On the boundary $\partial \Omega$ of $\Omega$ we consider the $(N-1)$-dimensional Hausdorff (surface) measure $\sigma$ and denote by $\Lprand{r}$ the boundary Lebesgue space with norm $\|\cdot\|_{r,\partial\Omega}$. We know that the trace mapping $W^{1,r}(\Omega) \to \Lprand{\tilde{r}}$ is compact for $\tilde{r}<r_*$ and continuous for $\tilde{r}=r_*$, where $r_*$ is the critical exponent of $r$ on the boundary given by
\begin{align*}
	r_*=
	\begin{cases}
		\frac{(N-1)r}{N-r} &\text{if }r<N,\\
		\text{any }\ell\in(r,\infty) & \text{if }r\geq N.
	\end{cases}
\end{align*}
For simplification we will avoid the notation of the trace operator throughout the paper.

Let us now introduce the Musielak-Orlicz Sobolev spaces. For this purpose, let $\mathcal{H}\colon \Omega \times [0,\infty)\to [0,\infty)$ be the function defined by
\begin{align*}
	(x,t)\mapsto t^p+\mu(x)t^q,
\end{align*}
where \eqref{conditions-exponents} is satisfied. Then, the Musielak-Orlicz space $L^\mathcal{H}(\Omega)$ is defined by
\begin{align*}
	L^\mathcal{H}(\Omega)=\left \{u ~ \Big | ~ u\colon \Omega \to \R \text{ is measurable and } \rho_{\mathcal{H}}(u)<\infty \right \}
\end{align*}
equipped with the Luxemburg norm
\begin{align*}
	\|u\|_{\mathcal{H}} = \inf \left \{ \tau >0 : \rho_{\mathcal{H}}\left(\frac{u}{\tau}\right) \leq 1  \right \},
\end{align*}
where the modular function $\rho_{\mathcal{H}}\colon\Lp{\mathcal{H}}\to \R$ is given by
\begin{align*}
	\rho_{\mathcal{H}}(u):=\into \mathcal{H}(x,|u|)\diff x=\into \big(|u|^{p}+\mu(x)|u|^q\big)\diff x.
\end{align*}
From Colasuonno-Squassina \cite[Proposition 2.14]{Colasuonno-Squassina-2016} we know that the space $L^\mathcal{H}(\Omega)$ is a reflexive Banach space.

Furthermore, we define the seminormed space
\begin{align*}
	L^q_\mu(\Omega)=\left \{u ~ \Big | ~ u\colon \Omega \to \R \text{ is measurable and } \into \mu(x) |u|^q \diff x< \infty \right \},
\end{align*}
which is endowed with the seminorm
\begin{align*}
	\|u\|_{q,\mu} = \left(\into \mu(x) |u|^q \diff x \right)^{\frac{1}{q}}.
\end{align*}
Similarly we define $L^q_\mu(\Omega;\R^N)$ with seminorm $\| F(\,\cdot\,)\|_{q,\mu}$.

The Musielak-Orlicz Sobolev space $W^{1,\mathcal{H}}(\Omega)$ is defined by
\begin{align*}
	W^{1,\mathcal{H}}(\Omega)= \left \{u \in L^\mathcal{H}(\Omega) : F(\nabla u) \in L^{\mathcal{H}}(\Omega) \right\}
\end{align*}
equipped with the norm
\begin{align*}
	\|u\|_{1,\mathcal{H},F}= \|F(\nabla u) \|_{\mathcal{H}}+\|u\|_{\mathcal{H}}.
\end{align*}

Finally, we mention the main embedding results between Musielak-Orlicz Sobolev spaces and usual Lebesgue and Sobolev spaces. We refer to Gasi\'nski-Winkert \cite[Proposition 2.2]{Gasinski-Winkert-2021} or Crespo-Blanco-Gasi\'nski-Harjulehto-Winkert \cite[Proposition 2.17]{Crespo-Blanco-Gasinski-Harjulehto-Winkert-2021}.

\begin{proposition}\label{proposition_embeddings}
	Let \eqref{conditions-exponents} be satisfied and let $p^*$ and $p_*$ be the critical exponents to $p$, see \eqref{critical_exponent}. Then the following embeddings hold:
	\begin{enumerate}
		\item[\textnormal{(i)}]
		$\Lp{\mathcal{H}} \hookrightarrow \Lp{r}$ and $\Wp{\mathcal{H}}\hookrightarrow \Wp{r}$ are continuous for all $r\in [1,p]$;
		\item[\textnormal{(ii)}]
		$\Wp{\mathcal{H}} \hookrightarrow \Lp{r}$ is continuous for all $r \in [1,p^*]$ and compact for all $r \in [1,p^*)$;
		\item[\textnormal{(iii)}]
		$\Wp{\mathcal{H}} \hookrightarrow \Lprand{r}$ is continuous for all $r \in [1,p_*]$ and compact for all $r \in [1,p_*)$;
		\item[\textnormal{(iv)}]
		$\Lp{\mathcal{H}} \hookrightarrow L^q_\mu(\Omega)$ is continuous;
		\item[\textnormal{(v)}]
		$\Lp{q}\hookrightarrow\Lp{\mathcal{H}} $ is continuous.
	\end{enumerate}
\end{proposition}

Let $B\colon \W\to \W^*$ be the nonlinear operator defined by
\begin{align}\label{operator_representation}
	\langle B(u),\ph\rangle_{\mathcal{H},F} :=\into \left(F^{p-1}(\nabla u)\nabla F(\nabla u) +\mu(x) F^{q-1}(\nabla u)\nabla F(\nabla u) \right)\cdot\nabla\ph \diff x,
\end{align}
where $\langle \cdot,\cdot\rangle_{\mathcal{H},F}$ is the duality pairing between $\W$ and its dual space $\W^*$. The operator $B\colon \W\to \W^*$ has the following properties, see Crespo-Blanco-Gasi\'nski-Harjulehto-Winkert \cite[Proposition 3.4(ii)]{Crespo-Blanco-Gasinski-Harjulehto-Winkert-2021}, by taking the properties of $F$ into account.

\begin{proposition}
	The operator $B$ defined by \eqref{operator_representation} is bounded, continuous and monotone (hence maximal monotone).
\end{proposition}

\section{Proof of the main result}

Let $J_\lambda\colon \W\to \R $ be the functional given by
\begin{align*}
	J_\lambda(u)&= \frac{1}{p}\| F(\nabla u)\|_p^p+\frac{1}{q}\| F(\nabla u)\|_{q,\mu}^q+\frac{1}{p}\|u\|_p^p+\frac{1}{q}\|u\|_{q,\mu}^q-\frac{1}{p^{*}}\|u_+\|_{p^*}^{p^{*}}\\
	&\quad  -\frac{\lambda}{\gamma}\int_{\Omega} \l(u_+\r)^\gamma\diff x-\lambda \int_{\Omega}G_1\l(x,u_+\r)\diff x-\frac{1}{p_{*}}\|u_+\|_{p_*,\partial\Omega}^{p_*}-\int_{\partial \Omega}G_2\l(x,u_+\r)\diff \sigma,
\end{align*}
where $u_{\pm}=\max(\pm u,0)$ and
\begin{align*}
	G_1(x,s)=\int^s_0 g_1(x,t)\diff t \quad\text{as well as}\quad
	G_2(x,s)=\int^s_0g_2(x,t)\diff t.
\end{align*}
Due to the presence of the singular term, it is easy to see that $J_\lambda$ is not $C^1$.

Throughout the paper we denote by $c_{p^{*}}$ and $c_{p_{*}}$ the inverses of the Sobolev embedding constants of $W^{1,p}(\Omega)\hookrightarrow L^{p^{*}}(\Omega)$ and $W^{1,p}(\Omega)\hookrightarrow L^{p_{*}}(\partial\Omega)$, respectively.
This means, in particular,
\begin{equation}\label{Sobolevineq}
	(c_{p^{*}})^{-1}=\inf_{\underset{u\neq0}{u\in W^{1,p}(\Omega),}}\frac{\|u\|_{1,p,F}}{\|u\|_{p^{*}}}
	\quad\text{and}\quad
	(c_{p_{*}})^{-1}=\inf_{\underset{u\neq0}{u\in W^{1,p}(\Omega),}}\frac{\|u\|_{1,p,F}}{\|u\|_{p_{*},\partial \Omega}}.
\end{equation}

Moreover, we define the function $\Psi\colon (0,\infty)\to \mathbb{R}$ given by
\begin{align}\label{function_psi}
	\Psi(s):=\frac{1}{p2^{p-1}r_F^p}-\frac{2^{p^*-1}c_{p^{*}}^{p^*}}{p^{*}}s^{p^*-p}-\frac{2^{p_*-1}c_{p_{*}}^{p_*}}{p_{*}}s^{p_*-p},
\end{align}
 where $\displaystyle r_F=\max_{w\neq 0}\frac{F(-w)}{F(w)}$ is finite by \textnormal{(H)(iii)}. Since $\Psi$ is strictly decreasing, we know there exists a unique $\varrho^*>0$ such that $\Psi(\varrho^*)=0$. In addition, $\Psi(s)\geq 0$ for all $s\in (0,\varrho^*)$. 
 
 We start with the study of the functional $I\colon \W\to \R $ given by
\begin{align*}
	I(u)&= \frac{1}{p}\| F(\nabla u)\|_p^p+\frac{1}{q}\| F(\nabla u)\|_{q,\mu}^q+\frac{1}{p}\|u\|_p^p+\frac{1}{q}\|u\|_{q,\mu}^q\\
	&\quad -\frac{1}{p^{*}}\|u\|_{p^*}^{p^{*}}-\frac{1}{p_{*}}\|u\|_{p_*,\partial\Omega}^{p_*}.
\end{align*}

The next proposition shows the  sequentially weakly lower semicontinuity of the functional  $I\colon \W\to \R $ on closed convex subsets of $\W$.

\begin{proposition}\label{proposition_semicontinuous}
	Let hypotheses \textnormal{(H)} be satisfied. For every $\varrho \in (0,\varrho^*)$ the restriction  of $I$ to the closed convex set $B_{\varrho}$ which is given by
	\begin{align*}
		B_{\varrho}:=\Big\{u\in \W \, : \, \|u\|_{1,p,F}\leq\varrho\Big\},
	\end{align*}
	is sequentially weakly lower semicontinuous.
\end{proposition}

\begin{proof} 
	Let $\varrho \in (0,\varrho^*)$ and let $\{u_{n}\}_{n\in\N}\subseteq B_{\varrho}$ be such that $u_{n}\weak u$ in $\W$. We are going to prove that
	\begin{align*}
		\liminf_{n\to\infty}\left(I(u_n)-I(u)\right)\geq 0.
	\end{align*}
	
	For $\kappa \geq 1$ we consider the following truncation functions $T_\kappa,R_\kappa\colon\R\to \R$ given by 
	\begin{align*}
		T_\kappa(s)&=
		\begin{cases}
			-\kappa&\text{if } s<-\kappa,\\
			s &\text{if } -\kappa \leq s \leq\kappa,\\
			\kappa &\text{if } s>\kappa,
		\end{cases}\qquad
		R_\kappa(s)=
		\begin{cases}
			s+\kappa &\text{if } s<-\kappa,\\
			0 &\text{if } -\kappa \leq s \leq \kappa,\\
			s-\kappa &\text{if } s>\kappa.
		\end{cases}
	\end{align*}
	Note that $T_\kappa(s)+R_\kappa(s)=s$ for all $s\in\R$.
	
	First,  we observe that
	\begin{align}\label{eq:1}
		\begin{split}
			\|F(\nabla u)\|_p^p
			&=\int_{\{|u|\leq k\}} F^p(\nabla u)\diff x+\int_{\{|u|> k\}} F^p(\nabla u)\diff x\\ &=\int_{\{|u|\leq k\}} F^p(\nabla (T_\kappa(u)))\diff x+\int_{\{|u|> k\}} F^p(\nabla (R_\kappa(u)))\diff x\\ 
			&=\|F(\nabla (T_\kappa(u)))\|_p^p+\|F(\nabla (R_\kappa(u)))\|_p^p.
		\end{split}
	\end{align}
	The same argument leads to
	\begin{align}\label{eq:2}
		\|F(\nabla u)\|_{q,\mu}^q
		=\|F(\nabla (T_\kappa(u)))\|_{q,\mu}^q+ \|F(\nabla (R_\kappa(u)))\|_{q,\mu}^q.
	\end{align}
	
	Since $\|\,\cdot\,\|_p$ is sequentially weakly lower semicontinuous and considering that $F(\nabla (T_\kappa(u_n)))\weak F(\nabla (T_\kappa(u)))$ in $L^q_\mu(\Omega)$ due to the weak convergence of $u_n\weak u$ in $\W$ we have for every $\kappa\geq 1$ 
	\begin{align}
		\begin{split}\label{eq:3}
			\liminf_{n \to \infty}\left(\frac{1}{p}\|F(\nabla (T_\kappa(u_n)))\|_p^p-\frac{1}{p} \|F(\nabla (T_\kappa(u)))\|_p^p\right)&\geq 0,\\
			\lim_{n \to \infty}\left(\frac{1}{q}\|F(\nabla (T_\kappa(u_n)))\|_{q,\mu}^q-\frac{1}{q}\| F(\nabla (T_\kappa(u)))\|_{q,\mu}^q \right)&=0.
		\end{split}	
	\end{align}

	Applying the triangle inequality for the Minkowski norm $F$, see Bao-Chern-Shen \cite[Theorem 1.2.2]{Bao-Chern-Shen-2000}, along with the convexity of the function $s\mapsto s^r$, $r>1$, we get the following inequality 
	\begin{equation}\label{eq:4}
		\frac{1}{2^{r-1}r_F^r}F^r(w_1-w_2)-2F^r(w_2)\leq F^r(w_1)-F^r(w_2)\quad\text{for all } w_1,w_2\in \R^N.
	\end{equation}
	From \eqref{eq:4} by taking $w_1=\nabla( R_\kappa(u_n))$ and $w_2=\nabla (R_\kappa(u))$, respectively, we get  
	\begin{align}\label{eq:5}
		\begin{split}
			&\|F (\nabla (R_\kappa(u_n)))\|_p^p-\|F(\nabla (R_\kappa(u)))\|_p^p\\
			&\geq \frac{1}{2^{p-1}r_F^p} \|F(\nabla (R_\kappa(u_n))-\nabla (R_\kappa(u)))\|_p^p-2 \|F(\nabla (R_\kappa(u)))\|_p^p.\\
			&\|F(\nabla (R_\kappa(u_n)))\|_{q,\mu}^q-\| F(\nabla (R_\kappa(u)))\|_{q,\mu}^q\\
			&\geq\frac{1}{2^{q-1}r_F^q} \|F(\nabla (R_\kappa(u_n))-\nabla (R_\kappa(u)))\|_{q,\mu}^q -2 \| F(\nabla (R_\kappa(u)))\|_{q,\mu}^q.
		\end{split}
	\end{align}
	On the other hand, by the Brezis-Lieb lemma, see, for example, Papageorgiou-Winkert \cite[Lemma 4.1.22]{Papageorgiou-Winkert-2018}, we have
	\begin{align}\label{eq:6}
		\begin{split}
			\liminf_{n\to\infty}\l(\|u_n\|_{p^*}^{p^*}-\|u\|_{p^*}^{p^*}\r)&=\liminf_{n\to\infty}\|u_n-u\|_{p^*}^{p^*},\\
			\liminf_{n\to\infty}\left(\|u_n\|_{p_*,\partial\Omega}^{p_*}-\|u\|_{p_*,\partial\Omega}^{p_*}\right)&=\liminf_{n\to\infty}\|u_n-u\|_{p_*,\partial\Omega}^{p_*}.
		\end{split}
	\end{align}

\vspace{0.2cm}
\noindent{\bf Claim:} $\|h\|_p^p \geq \|R_\kappa(h)\|_p^p$ for all $h \in \W$ and for all $\kappa\geq1$.

	First, we have
	\begin{align}\label{aux_1}
		\begin{split}
			\|h\|_p^p 
			& =\|T_\kappa(h)+R_\kappa(h)\|_p^p\\
			& =\int_{\{h<-\kappa\}} |-\kappa+R_\kappa(h)|^p\diff x+\int_{\{|h|\leq \kappa\}} |u+R_\kappa(h)|^p\diff x\\
			&\quad + \int_{\{h>\kappa\}} |\kappa+R_\kappa(h)|^p\diff x\\
			&\geq \int_{\{h<-\kappa\}} |-\kappa+R_\kappa(h)|^p\diff x+\int_{\{h>\kappa\}} |R_\kappa(h)|^p\diff x.
		\end{split}
	\end{align}
	Applying the inequality
	\begin{align*}
		|w_2|^p>|w_1|^p+p|w_1|^{p-2}w_1(w_2-w_1)\quad\text{for all }w_1, 	w_2 \in\R^N,
	\end{align*}
	with $w_2=R_\kappa(h)-\kappa$ and $w_1=R_\kappa(h)$, we get
	\begin{align}\label{aux_2}
		\begin{split}
			&\int_{\{h<-\kappa\}} |-\kappa+R_\kappa(h)|^p\diff x\\
			& \geq\int_{\{h<-\kappa\}}\l[ |R_\kappa(h)|^p+p |R_\kappa(h)|^{p-2}R_\kappa(h)\cdot (-\kappa)\r]\diff x\\
			& \geq \int_{\{h<-\kappa\}}|R_\kappa(h)|^p\diff x
		\end{split}	
	\end{align}
	since $R_\kappa(h)<0$ if $h<-\kappa$. Combining \eqref{aux_1} and \eqref{aux_2} leads to
	\begin{align*}
		\|u\|_p^p 
		&\geq \int_{\{h<-\kappa\}}|R_\kappa(h)|^p\diff x+\int_{\{h>\kappa\}} |R_\kappa(h)|^p\diff x=\|R_\kappa(h)\|_p^p
	\end{align*}
	because $R_\kappa(h)=0$ if $|h| \leq \kappa$. This proves the Claim.
	
	Thus, we may apply the Brezis-Lieb lemma along with the Claim in order to obtain 
	\begin{align}\label{eq:7}
		\begin{split}
			\liminf_{n\to \infty }\left(\|u_n\|_p^p-\|u\|_p^p\right)
			&=\liminf_{n\to \infty }\|u_n-u\|_p^p\\
			&\geq \liminf_{n\to \infty }\|R_\kappa(u_n)-R_\kappa(u)\|_p^p\\ &\geq \frac{1}{2^{p-1}r_F^p}\liminf_{n\to \infty } \|R_\kappa(u_n)-R_\kappa(u)\|_p^p
		\end{split}
	\end{align}
	since $r_F \geq 1$ and so $2^{p-1}r_F^p \geq 1$.
	
	Note that 
	\begin{align}\label{eq:8}
		\begin{split}
			\|F(\nabla (R_\kappa(u)))\|_p^p &\to 0 \quad\text{as }\kappa\to\infty,\\\
			\|F(\nabla(R_\kappa(u)))\|_{q,\mu}^q&\to 0 \quad\text{as }\kappa\to\infty,\\
			\|u_n\|_{q,\mu}^q &\to \|u\|_{q,\mu}^q\quad\text{as }n\to\infty.
		\end{split}	
	\end{align}
	The last convergence in \eqref{eq:8} follows from Proposition \ref{proposition_embeddings} (ii) since $q<p^*$ and due to the boundedness of $\mu(\cdot)$, as given in \eqref{conditions-exponents}.
	
	Hence, for $\kappa$ large enough, taking \eqref{eq:1}, \eqref{eq:2}, \eqref{eq:3}, \eqref{eq:5}, \eqref{eq:6}, \eqref{eq:7} and \eqref{eq:8} into account, we have that 
	\begin{align}\label{eq:9}
		\begin{split}
			&\liminf_{n\to\infty}\left(I(u_{n})-I(u)\right)\\
			& \geq\liminf_{n\to\infty}\left(\frac{1}{p2^{p-1}r_F^p}\|R_\kappa(u_{n})-R_\kappa(u)\|_{1,p,F}^{p}\r.\\
			&\l.\qquad\qquad\quad -\frac{1}{p^{*}}\|u_n-u\|_{p^*}^{p^*}-\frac{1}{p_{*}}\|u_n-u\|_{p_*,\partial\Omega}^{p_*}\r).
		\end{split}
	\end{align}
	We observe that
	\begin{align}\label{eq:10}
		\begin{split}
			&\|u_n-u\|_{p^*}^{p^*}\\
			&\leq 2^{p^*-1}\|T_\kappa(u_{n})-T_\kappa(u)\|_{p^*}^{p^*}+2^{p^*-1}\|R_\kappa(u_{n})-R_\kappa(u)\|_{p^*}^{p^*}, \\ 
			&\|u_n-u\|_{p_*,\partial\Omega}^{p_*}\\
			&\leq 2^{p_*-1}\|T_\kappa(u_{n})-T_\kappa(u)\|_{p_*,\partial\Omega}^{p_*}+2^{p_*-1}\|R_\kappa(u_{n})-R_\kappa(u)\|_{p_*,\partial\Omega}^{p_*}.
		\end{split}
	\end{align}
	By Lebesgue's dominated convergence theorem we get that 
	\begin{align}\label{eq:11}
		\begin{split}
			&\lim_{n \to \infty} \|T_\kappa(u_{n})-T_\kappa(u)\|_{p^*}^{p^*}= 0\quad\text{and}\quad\lim_{n \to \infty} \|T_\kappa(u_{n})-T_\kappa(u)\|_{p_*,\partial\Omega}^{p_*}= 0.
		\end{split}
	\end{align}
	Finally, combining \eqref{eq:9}, \eqref{eq:10} and \eqref{eq:11} we arrive at
	\begin{align*}
		&\liminf_{n\to\infty}\left(I(u_{n})-I(u)\right)\\ &\geq\liminf_{n\to\infty}\Bigg(\frac{1}{p2^{p-1}r_F^p}\|R_\kappa(u_{n})-R_\kappa(u)\|_{1,p,F}^{p}\\
		&\qquad\qquad\quad -\frac{2^{p^*-1}}{p^{*}}\|R_\kappa(u_{n})-R_\kappa(u)\|_{p^*}^{p^*}-\frac{2^{p_*-1}}{p_{*}}\|R_\kappa(u_{n})-R_\kappa(u)\|_{p_*,\partial\Omega}^{p_*} \Bigg).
	\end{align*}
	Using this along with \eqref{Sobolevineq} and the fact that $\psi(s)\geq 0$ for all $s\in (0,\varrho^*)$ (see \eqref{function_psi}), it follows that
	\begin{align*}
		&\liminf_{n\to\infty}\left(I(u_{n})-I(u)\right)\\ &\geq\liminf_{n\to\infty}\Bigg(\frac{1}{p2^{p-1}r_F^p}\|R_\kappa(u_{n})-R_\kappa(u)\|_{1,p,F}^{p}\\
		&\qquad\qquad \quad -\frac{2^{p^*-1}c_{p^{*}}^{p^*}}{p^{*}}\|R_\kappa(u_n)-R_\kappa(u)\|_{1,p,F}^{p^*}\\
		&\qquad\qquad \quad
		-\frac{2^{p_*-1}c_{p_{*}}^{p_*}}{p_{*}}\|R_\kappa(u_n)-R_\kappa(u)\|_{1,p,F}^{p_*} \Bigg)\\
		&\geq \liminf_{n\to\infty}\Bigg(\|R_\kappa(u_{n})-R_\kappa(u)\|_{1,p,F}^{p}\Psi(\varrho)\Bigg)\geq 0,
	\end{align*}
	which proves the assertion of the proposition. 
\end{proof}

Taking into account the assumption \textnormal{(H)(ii)}  together with the compact embeddings $\W$ $\hookrightarrow$ $L^{r_1}(\Omega)$ for $r_1<p^*$ and $\W\hookrightarrow L^{r_2}(\partial \Omega)$ for $r_2<p_*$, see Proposition \ref{proposition_embeddings} (ii), (iii), it is quite standard to prove that the functional 
\begin{align*}
	u\mapsto \frac{\lambda}{\gamma}\int _{\Omega}\l(u_+\r)^{\gamma}\diff x+\lambda\int _{\Omega}G(x,u_+)\diff x+ \int _{\partial \Omega}G_2(x,u_+)\diff \sigma,
\end{align*}
is sequentially weakly lower semicontinuous on $\W$ for every $\lambda>0$. This fact along with Proposition \ref{proposition_semicontinuous} leads to the following corollary.

\begin{corollary}\label{corollary_semicontinuous}
	Let hypotheses \textnormal{(H)} be satisfied. For every $\lambda>0$ and for every $\varrho \in (0,\varrho^*)$ the restriction  of $J_\lambda$ to the closed convex set $B_{\varrho}$ is sequentially weakly lower semicontinuous.
\end{corollary}

Now we are going to prove Theorem \ref{main_theorem}. For this purpose, we introduce the functionals $I_1\colon\W\to\R$ and $I_2\colon \Lp{\mathcal{H}}\to\R$ given by
\begin{align*}
	{I_1}(u)&=-\frac{1}{q}\|F(\nabla u)\|_{q,\,u}^q-\frac{1}{q}\|u\|_{q,\mu}^q+\frac{1}{p^{*}}\|u_+\|_{p^*}^{p^*}+\frac{\lambda}{\gamma}\int _{\Omega}\l(u_+\r)^{\gamma}\diff x\\
	&\quad +\lambda \int _{\Omega}G_1(x,u_+)\diff x+\frac{1}{p_{*}}\|u_+\|_{p_*,\partial\Omega}^{p_*}+\int _{\partial \Omega}G_2(x,u_+)\diff \sigma
\end{align*}
and 
\begin{align*} 
	I_2(u)
	&=\frac{1}{p^{*}}\|u_+\|_{p^*}^{p^*}+\frac{\lambda}{\gamma}\int _{\Omega}\l(u_+\r)^{\gamma}\diff x+\lambda\int _{\Omega}G(x,u_+)\diff x\\
	&\quad +\frac{1}{p_{*}}\|u_+\|_{p_*,\partial\Omega}^{p_*}+ \int _{\partial \Omega}G_2(x,u_+)\diff \sigma.
\end{align*}

\begin{proof}[Proof of Theorem \ref{main_theorem}] 
	Let $\lambda>0$ and $\varrho \in (0,\varrho^*)$ be as in Corollary \ref{corollary_semicontinuous}. First we define 
	\begin{align*}
		\varphi_{\lambda}(\varrho):=\inf_{\|u\|_{1,p,F}<\varrho}\frac{\sup_{B_{\varrho}}I_1-I_1(u)}{\varrho^{p}-\|u\|_{1,p,F}^{p}}
		\quad\text{and}\quad
		\psi_{\lambda}(\varrho):=\sup_{B_{\varrho}}I_1.
	\end{align*}
	
	\vspace{0.2cm}
	\noindent{\bf Claim:} {\em There exist $\lambda$, $\varrho>0$ small enough such that}
	\begin{equation}
		\varphi_{\lambda}(\varrho)<\frac{1}{p}.\label{min}
	\end{equation}

	In order to prove \eqref{min}, it is enough to find $\lambda$, $\varrho>0$ such that 
	\begin{equation}\label{eq:whatweeneed}
		\inf_{\xi<\varrho}\frac{\psi_{\lambda}(\varrho)-\psi_{\lambda}(\xi)}{\varrho^{p}-\xi^{p}}<\frac{1}{p}.
	\end{equation}
	Taking $\xi=\varrho-\varepsilon$ for some $\eps \in (0,\varrho)$ we easily see that
	\begin{align*}
		\frac{\psi_{\lambda}(\varrho)-\psi_{\lambda}(\xi)}{\varrho^{p}-\xi^{p}} & 
		= \frac{\psi_{\lambda}(\varrho)-\psi_{\lambda}(\varrho-\varepsilon)}{\varrho^{p}-(\varrho-\varepsilon)^{p}}\\
		& = \frac{\psi_{\lambda}(\varrho)-\psi_{\lambda}(\varrho-\varepsilon)}{\varepsilon}\cdot\frac{-\frac{\varepsilon}{\varrho}}{\varrho^{p-1}[(1-\frac{\varepsilon}{\varrho})^{p}-1]}.
	\end{align*}
	Therefore, if we pass to the limit as $\eps\to 0$, then \eqref{eq:whatweeneed} holds if 
	\begin{equation}\label{eq:ezkell}
		\limsup_{\varepsilon\to0^{+}}\frac{\psi_{\lambda}(\varrho)-\psi_{\lambda}(\varrho-\varepsilon)}{\varepsilon}<\varrho^{p-1}
	\end{equation}
	is satisfied.
	
	Thus we have to verify \eqref{eq:ezkell} to get our Claim.
	First note that
	\begin{align*}
		&\frac{1}{\varepsilon}\left|\psi_{\lambda}(\varrho)-\psi_{\lambda}(\varrho-\varepsilon)\right|\\
		& =\frac{1}{\varepsilon}\left|\sup_{v\in B_{1}}I_1(\varrho v)-\sup_{v\in B_{1}}I_1((\varrho-\varepsilon)v)\right|\\
		& \leq\frac{1}{\varepsilon}\sup_{v\in B_{1}}\left|I_1(\varrho v)-I_1((\varrho-\varepsilon)v)\right|\\
		& \leq\frac{1}{\varepsilon}\sup_{v\in B_{1}}\left|\frac{(\varrho-\varepsilon)^{q}-\varrho^{q}}{q}\Big[\|F(\nabla v)\|_{q,\mu}^q+\|v\|_{q,\mu}^q\Big]+I_2(\varrho v)-I_2((\varrho-\varepsilon)v)\right|.
	\end{align*}
	The growth conditions in \textnormal{(H)(ii)}, along with the continuous embeddings $\Wp{p}\hookrightarrow L^{p^{*}}(\Omega)$ as well as $\Wp{p}\hookrightarrow L^{p_{*}}(\partial\Omega)$, yield
	\begin{align*}
		&\frac{\psi_{\lambda}(\varrho)-\psi_{\lambda}(\varrho-\varepsilon)}{\varepsilon}\\
		&\leq  \frac{1}{\varepsilon}\sup_{\|v\|_{1,{p},F}\leq1}\int_{\Omega}\l|\int_{(\varrho-\varepsilon)v_+(x)}^{\varrho v_+(x)}\l[ t^{p^{*}-1}+\lambda t^{\gamma-1}+\lambda g_1(x,t)\r]\diff t\r|\diff x\\
		& \quad+\frac{1}{\varepsilon}\sup_{\|v\|_{1,{p},F}\leq1}\int_{\partial\Omega}\left|\int_{(\varrho-\varepsilon)v_+(x)}^{\varrho v_+(x)}\left[t^{p_*-1}+g_2(x,t)\right]\diff t\right|\diff \sigma\\
		& \leq \frac{c_{p^{*}}^{p^{*}}}{p^{*}}\l|\frac{\varrho^{p^{*}}-(\varrho-\varepsilon)^{p^{*}}}{\varepsilon}\r|+\lambda \frac{c_{p^{*}}^{\gamma}|\Omega|^{\frac{p^{*}-\gamma}{p^{*}}}}{\gamma}\l|\frac{\varrho^{\gamma}-(\varrho-\varepsilon)^{\gamma}}{\varepsilon}\r|\\
		& \quad +\lambda a_{1}\frac{c_{p^{*}}^{\nu_1}|\Omega|^{\frac{p^{*}-\nu_1}{p^{*}}}}{\nu_1}\l|\frac{\varrho^{\nu_1}-(\varrho-\varepsilon)^{\nu_1}}{\varepsilon}\r|
		+  \lambda b_{1}\frac{c_{p^{*}}^{\theta_1}|\Omega|^{\frac{p^{*}-\theta_1}{p^{*}}}}{\theta_1}\l|\frac{\varrho^{\theta_1}-(\varrho-\varepsilon)^{\theta_1}}{\varepsilon}\r| \\  
		&\quad +\frac{c_{p_{*}}^{p_{*}}}{p_{*}}\l|\frac{\varrho^{p_{*}}-(\varrho-\varepsilon)^{p_{*}}}{\varepsilon}\r|+  a_2 \frac{c_{p_{*}}^{\nu_2}|\partial\Omega|^{\frac{p_{*}-\nu_2}{p_{*}}}}{\nu_2}\l|\frac{\varrho^{\nu_2}-(\varrho-\varepsilon)^{\nu_2}}{\varepsilon}\r|.
	\end{align*}
	Hence, we obtain
	\begin{align*}
		&\limsup_{\varepsilon\rightarrow 0^+}\frac{\psi_{\lambda}(\varrho)-\psi_{\lambda}(\varrho-\varepsilon)}{\varepsilon}\\
		&\leq  c_{p^{*}}^{p^{*}}\varrho^{p^{*}-1}+\lambda c_{p^{*}}^{\gamma}|\Omega|^{\frac{p^{*}-\gamma}{p^{*}}}\varrho^{\gamma-1}+\lambda a_{1}c_{p^{*}}^{\nu_1}|\Omega|^{\frac{p^{*}-\nu_1}{p^{*}}}\varrho^{\nu_1-1}\\
		&\quad +\lambda b_{1}c_{p^{*}}^{\theta_1}|\Omega|^{\frac{p^{*}-\theta_1}{p^{*}}}\varrho^{\theta_1-1}+c_{p_{*}}^{p_{*}}\varrho^{p_{*}-1}+  a_{2}c_{p_{*}}^{\nu_2}|\partial\Omega|^{\frac{p_{*}-\nu_2}{p_{*}}}\varrho^{\nu_2-1}.
	\end{align*}
	
	Now we consider the function $\Lambda\colon(0,\infty)\to \R$ given by
	\begin{align*}
		\Lambda(s)=\frac{s^{p-\gamma}-c_{p^{*}}^{p^*}s^{p^*-\gamma}-c_{p_{*}}^{p_*}s^{p_*-\gamma}-a_2 c_{p_{*}}^{\nu_2} |\partial \Omega|^\frac{p_*-\nu_2}{p_*}s^{\nu_2-\gamma}}{c_{p^{*}}^\gamma |\Omega|^\frac{p^*-\gamma}{p^*}+a_1 c_{p^{*}}^{\nu_1} |\Omega|^\frac{p^*-\nu_1}{p^*}s^{\nu_1-\gamma}+b_1c_{p^{*}}^{\theta_1} |\Omega|^\frac{p^*-\theta_1}{p^*}s^{\theta_1-\gamma}}.
	\end{align*}
	
	We easily see that $\displaystyle \lim_{s\to 0}\Lambda(s)=0$ and from L'Hospital's rule we verify that $\displaystyle \lim_{s\to \infty}\Lambda(s)=-\infty$. Moreover, since $\nu_2>p$, see \textnormal{(H)(ii)}, and due to the continuity of $\Lambda$, we know that there exists $s_0>0$ small enough such that $\Lambda(s)>0$ for all $s\in (0,s_0)$. Hence, we find $s_{\max}>0$ such that 
	\begin{align*}
		\Lambda(s_{\max})=\max_{s>0}\Lambda(s).
	\end{align*}
	Let us set
	\begin{align*}
		\lambda_*:=\Lambda\left(\min\{s_{\rm max},\varrho^*\}\right).
	\end{align*}
	If we now take $\lambda<\lambda_*$ and $\varrho<\min\{s_{\rm max},\varrho^*\}$, then \eqref{eq:ezkell} is satisfied and so \eqref{min}. This proves the Claim.
	
	From the Claim we know that there exists an element $\hat{u}\in\W$ with $\|\hat{u}\|_{1,p,F}\leq \varrho$ such that 
	\begin{align}\label{eq:u0}
		J_\lambda(\hat{u})<\frac{1}{p}\varrho^{p}-I_1(u_1)\quad\text{for all }u_1\in B_{\varrho}.
	\end{align}
	From Corollary \ref{corollary_semicontinuous} we know that $J_{\lambda}\big|_{B\varrho}$ is sequentially weakly lower semicontinuous. Therefore, $J_\lambda\colon \W\to\R$ restricted to $B_\varrho$ has a global minimizer $u \in \W$ with $\|u\|_{1,{p},F}\leq \varrho$. Suppose that $\|u\|_{1,p,F}=\varrho$. Then we have from \eqref{eq:u0} that
	\begin{align*}
		J_\lambda(u)=\frac{1}{p}\varrho^{p}-I_1(u)>J_{\lambda}(\hat{u}),
	\end{align*}
	which is a contradiction. We conclude that $u\in B_\varrho$ is a local minimizer for $J_{\lambda}$ with $\|u\|_{1,{p},F}<\varrho$ for $\lambda<\lambda_*$. 

	We claim that $u\neq 0$. Let $v\in\W$ be such that $v>0$ and let $t>0$. Then we have
	\begin{align*}
		J_\lambda(tv)&= \frac{t^p}{p}\| F(\nabla v)\|_p^p+\frac{t^q}{q}\| F(\nabla v)\|_{q,\mu}^q+\frac{t^p}{p}\|v\|_p^p+\frac{t^q}{q}\|v\|_{q,\mu}^q-\frac{t^{p^*}}{p^{*}}\|v\|_{p^*}^{p^{*}}\\
		&\quad  -\lambda\frac{t^\gamma}{\gamma}\int_{\Omega} v^\gamma\diff x- \lambda\frac{a_1t^{\nu_1} }{\nu_1}\|v\|_{\nu_1}^{\nu_1}- \lambda\frac{b_1t^{\theta_1} }{\theta_1}\|v\|_{\theta_1}^{\theta_1}\\
		&\quad -\frac{  t^{p_*}}{p_*}\|v\|^{p_*}_{p_*,\partial\Omega}-\frac{a_2t^{\nu_2} }{\nu_2}\|v\|^{\nu_2}_{\nu_2,\partial\Omega},
	\end{align*}
	which implies $J_\lambda(t v) <0$ for $t>0$ sufficiently small. Thus, $u \neq 0$.
	
	Let us now prove that $u\in\W$ is nonnegative a.\,e.\,in $\Omega$. First we observe that $u+t u_{-}\in B_\varrho$ and $(u+tu_-)_+=u_+$ for $t>0$ sufficiently small. Using this fact we have
	\begin{align*}
		0
		&\leq \frac{J_\lambda(u+tu_{-})- J_\lambda(u)}{t}\\ 
		&=\frac{1}{p}\into \frac{ F^p(\nabla (u+tu_{-}))-F^p(\nabla u)}{t}\diff x\ +\frac{1}{q}\int_\Omega \mu(x)\frac{F^q(\nabla(u+tu_-))-F^q(\nabla u)}{t}\diff x\\
		&\quad+\frac{1}{p}\into\frac{|u+tu_-|^p-|u|^p}{t} \diff x+\frac{1}{q}\int_\Omega
		\mu(x)\frac{|u+tu_-|^q-|u|^q}{t}\diff x.
	\end{align*}
	From this we conclude
	\begin{align*}
		0& \leq \lim_{t\to 0^+}\frac{J_\lambda(u+tu_{-})- J_\lambda(u)}{t}\\ &=\int_\Omega F^{p-1}(\nabla u)\nabla F( \nabla u)\cdot\nabla u_{-}\diff x +\int_\Omega \mu(x)F^{q-1}(\nabla u)\nabla F(\nabla u)\cdot\nabla u_-\diff x\\
		&\quad +\int_\Omega |u|^{p-2}uu_{-} \diff x +\int_\Omega \mu(x)|u|^{q-2}uu_{-} \diff x.
	\end{align*}
	However, from Proposition \ref{basic-properties} \textnormal{(iii)}, we know that 
	\begin{align*}
		\int_\Omega F(\nabla u)^{p-1}\nabla F( \nabla u)\cdot\nabla u_{-}\diff x
		&=-\int_\Omega F^{p-1}(\nabla u_-)\nabla F( \nabla u_-)\cdot\nabla u_{-}\diff x\\ 
		&=-\|F(\nabla u_-)\|_p^p
	\end{align*}
	and
	\begin{align*}
		\int_\Omega \mu(x)F^{q-1}(\nabla u)\nabla F(\nabla u)\cdot\nabla u_-\diff x=-\|F(\nabla u_-)\|_{q,\mu}^q.
	\end{align*}
	This leads to
	\begin{align*}
		0&\leq \lim_{t\to 0} \frac{J_\lambda(u+tu_{-})- J_\lambda(u)}{t}\\
		&=-\|F(\nabla u_-)\|_p^p-\|F(\nabla u_-)\|_{q,\mu}^q-\|u_-\|_p^p-\|u_-\|_{q,\mu}^q \leq 0.
	\end{align*}
	Therefore, $u_{-}=0$ and so $u\geq 0$ a.\,e.\,in $\Omega$.

	Let us now show that $u$ is positive in $\Omega$. We argue indirectly and suppose there is a set $C$ with positive measure such that $u=0$ in $C$.  Let $\varphi \in \W$ with $\ph>0$ and let $t>0$ small enough such that $u+t\varphi \in B_\sigma$ and $(u+t\varphi)^\gamma>u^\gamma$ a.\,e.\,in $\Omega$.  We obtain
	\begin{align*}
		\begin{split}
			0& \leq \frac{J_\lambda (u+t\varphi)- J_\lambda(u)}{t}\\
			& = \frac{1}{p}\frac{\|F(\nabla (u+t\varphi))\|^p_{p}-\|F(\nabla u)\|^p_{p}}{t} +\frac{1}{q}\frac{\|F(\nabla (u+t\varphi))\|^q_{q,\mu}-\|F(\nabla u)\|^q_{q,\mu}}{t} \\ 
			&\quad 
			+\frac{1}{p}\frac{\|u+t\varphi\|^p_{p}-\|u\|^p_{p}}{t} +\frac{1}{q}\frac{\|u+t\varphi\|^q_{q,\mu}-\|u\|^q_{q,\mu}}{t} -\frac{1}{p^*}\frac{\|u+t\varphi\|^{p^*}_{p^*}-\|u\|^{p^*}_{p^*}}{t}\\ 
			&\quad -\frac{\lambda}{\gamma t^{1-\gamma}}\int_C \varphi^\gamma\diff x -\frac{\lambda}{\gamma}\int_{\Omega\setminus C} \frac{(u+t\varphi)^{\gamma}-u^{\gamma}}{t}\diff x\\
			&\quad - \lambda\int_\Omega \frac{G_1(x,u+t\varphi)-G_1(x,u)}{t}\diff x \\
			&\quad -\frac{1}{p_*}\frac{\|u+t\varphi\|^{p_*}_{p_*,\partial\Omega}-\|u\|^{p_*}_{p_*,\partial\Omega }}{t} - \int_{\partial \Omega} \frac{G_2(x,u+t\varphi)-G_2(x,u)}{t}\diff \sigma \\
			&< \frac{1}{p}\frac{\|F(\nabla (u+t\varphi))\|^p_{p}-\|F(\nabla u)\|^p_{p}}{t} +\frac{1}{q}\frac{\|F(\nabla (u+t\varphi))\|^q_{q,\mu}-\|F(\nabla u)\|^q_{q,\mu}}{t} \\ 
			&\quad 
			+\frac{1}{p}\frac{\|u+t\varphi\|^p_{p}-\|u\|^p_{p}}{t} +\frac{1}{q}\frac{\|u+t\varphi\|^q_{q,\mu}-\|u\|^q_{q,\mu}}{t} -\frac{1}{p^*}\frac{\|u+t\varphi\|^{p^*}_{p^*}-\|u\|^{p^*}_{p^*}}{t}\\ 
			&\quad -\frac{\lambda}{\gamma t^{1-\gamma}}\int_C \varphi^\gamma\diff x 
			- \lambda\int_\Omega \frac{G_1(x,u+t\varphi)-G_1(x,u)}{t}\diff x \\
			&\quad -\frac{1}{p_*}\frac{\|u+t\varphi\|^{p_*}_{p_*,\partial\Omega}-\|u\|^{p_*}_{p_*,\partial\Omega }}{t}- \int_{\partial \Omega} \frac{G_2(x,u+t\varphi)-G_2(x,u)}{t}\diff \sigma.
		\end{split}
	\end{align*}
	This yields
	\begin{align*}
		\begin{split}
			0& \leq \frac{J_\lambda (u+t\varphi)- J_\lambda(u)}{t}
			\longrightarrow -\infty \quad \text{as} \quad  t\to 0^+,
		\end{split}
	\end{align*}
	a contradiction. Hence, $u>0$ a.\,e.\,in $\Omega$.

	Next we want to show that
	\begin{equation}\label{def1}
		u^{\gamma-1}\varphi\in L^1(\Omega) \quad \text{for all } \varphi\in\W
	\end{equation}
	and
	\begin{align}\label{def2}
		\begin{split}
			& \int_{\Omega}\left(F(\nabla u)^{p-1}+\mu(x)F(\nabla u)^{q-1}\right)\nabla F(\nabla u)\cdot\nabla\varphi\diff x + \int_{\Omega} u^{p-1}\ph \diff x\\ 
			& +\int_{\Omega} \mu(x) u^{q-1}\ph\diff x-\int_{\Omega}u^{p^{*}-1}\varphi\diff x -\lambda\int_{\Omega}u^{\gamma-1}\varphi\diff x-\lambda\int_{\Omega}g_1(x,u)\varphi\diff x\\
			&-\int_{\partial\Omega} u^{p_*-1}\ph \diff \sigma - \int_{\partial\Omega} g_2(x,u)\ph \diff \sigma\geq 0\quad 
		\end{split}
	\end{align}
	for all $\ph \in \W$ with $\ph \geq 0$.

	We choose $\varphi\in \W$ with $\varphi\geq 0$ and fix a decreasing sequence $\{t_n\}_{n \in\N} \subseteq (0,1]$ such that $\displaystyle \lim_{n\to \infty} t_n=0$.  It is clear that the functions
	\begin{align*}
		h_n(x)=\frac{(u(x)+t_n\varphi(x))^\gamma-u(x)^\gamma}{t_n}, \quad n\in\N
	\end{align*}
	are measurable and nonnegative. Moreover, we have
	\begin{align*}
		\lim_{n\to \infty} h_n(x)=\gamma u(x)^{\gamma-1}\varphi(x)\quad \text{for a.\,a.\,} x\in\Omega.
	\end{align*}
	Applying Fatou's lemma gives
	\begin{equation}\label{fatou}
		\int_\Omega u^{\gamma-1}\varphi\diff x\leq \frac{1}{\gamma}
		\liminf_{n\to\infty}\int_\Omega h_n\diff x.
	\end{equation}
	Then, for $n\in\N$ large enough, we obtain
	\begin{align*}
		\begin{split}
			0& \leq \frac{J_\lambda (u+t\varphi)- J_\lambda(u)}{t}\\
			& =  \frac{1}{p}\frac{\|F(\nabla (u+t_n\varphi))\|^p_{p}-\|F(\nabla u)\|^p_{p}}{t_n} +\frac{1}{q}\frac{\|F(\nabla (u+t_n\varphi))\|^q_{q,\mu}-\|F(\nabla u)\|^q_{q,\mu}}{t_n} \\ 
			&\quad 
			+\frac{1}{p}\frac{\|u+t_n\varphi\|^p_{p}-\|u\|^p_{p}}{t_n} +\frac{1}{q}\frac{\|u+t_n\varphi\|^q_{q,\mu}-\|u\|^q_{q,\mu}}{t_n} -\frac{1}{p^*}\frac{\|u+t_n\varphi\|^{p^*}_{p^*}-\|u\|^{p^*}_{p^*}}{t_n}\\ 
			&\quad -\frac{\lambda}{\gamma }\int_\Omega h_n\diff x 
			- \lambda\int_\Omega \frac{G_1(x,u+t_n\varphi)-G_1(x,u)}{t_n}\diff x \\
			&\quad -\frac{1}{p_*}\frac{\|u+t_n\varphi\|^{p_*}_{p_*,\partial\Omega}-\|u\|^{p_*}_{p_*,\partial\Omega }}{t_n}- \int_{\partial \Omega} \frac{G_2(x,u+t_n\varphi)-G_2(x,u)}{t_n}\diff \sigma.
		\end{split}
	\end{align*}
	Passing to the limit as $n\to \infty$ in the inequality above and using \eqref{fatou} we derive  \eqref{def1} and we have
	\begin{align*}
		\lambda\int_\Omega u^{\gamma-1}\varphi\diff x
		& \leq\int_{\Omega}\left(F(\nabla u)^{p-1}+\mu(x)F(\nabla u)^{q-1}\right)\nabla F(\nabla u)\cdot \nabla\varphi\diff x\\
		&\quad + \int_{\Omega} u^{p-1}\ph \diff x+\int_{\Omega} \mu(x) u^{q-1}\ph\diff x -\int_{\Omega}u^{p^{*}-1}\varphi\diff x\\
		&\quad -\lambda\int_{\Omega}u^{\gamma-1}\varphi\diff x-\lambda\int_{\Omega}g_1(x,u)\varphi\diff x\\
		&\quad -\int_{\partial\Omega} u^{p_*-1}\ph \diff \sigma - \int_{\partial\Omega} g_2(x,u)\ph \diff \sigma,
	\end{align*}
	which shows \eqref{def2}. Note that it is sufficient to prove the integrability in  \eqref{def1} for nonnegative
test functions $\ph\in \W$.

	Now, let $\varepsilon\in (0,1)$ be such that $(1+t)u\in B_\sigma$ for all $t\in[-\varepsilon, \varepsilon]$. Note that the function $\beta(t):=J_\lambda((1+t)u)$ has a local minimum in zero. We apply again Proposition \ref{basic-properties} \textnormal{(iii)} in order to get
	\begin{align}\label{beta}
		\begin{split}
			0&=\beta'(0)=\lim_{t\to 0}\frac{J_\lambda((1+t)u)-J_\lambda(u)}{t}\\
			& = \| F(\nabla u)\|_p^p+\| F(\nabla u)\|_{q,\mu}^q +\|u\|_p^p+\|u\|_{q,\mu}^q-\|u\|_{p^*}^{p^*}\\
			&\quad -\lambda\int_\Omega u^{\gamma}\diff x-\lambda\int_\Omega g_1(x,u)u\diff x-\|u\|_{p_*,\partial\Omega}^{p_*}-\int_{\partial \Omega}g_2\l(x,u\r)u\diff \sigma.
		\end{split}
	\end{align}
	Finally, we need to show that $u$ is a positive weak solution of \eqref{problem}. To this end, let $v\in\W$ and take the test function $\ph=(u+\eps v)_+\in\W$ in \eqref{def2}. Taking \eqref{beta} into account we have
	\begin{align}\label{last}
		0
		&\leq \int_{\{u+\eps v\geq 0\}}\left(F^{p-1}(\nabla u)+\mu(x)F^{q-1}(\nabla u)\right)\nabla F(\nabla u)\cdot \nabla(u+\eps v)\diff x
		\nonumber\\
		&\quad
		+\int_{\{u+\eps v\geq 0\}}\left(u^{p-1}+\mu(x)u^{q-1}\right ) (u+\eps v)\diff x-\int_{\Omega}u^{p^{*}-1}(u+\eps v)\diff x
		\nonumber\\
		& \quad  
		-\lambda\int_{\Omega}u^{\gamma-1}(u+\eps v)\diff x-\lambda\int_{\Omega}g_1(x,u)(u+\eps v)\diff x
		\nonumber\\
		&\quad 
		-\int_{\partial\Omega}\l(u^{p_*-1}+g_2(x,u)\r)(u+\eps v)\diff \sigma
		\nonumber\\
		&= \| F(\nabla u)\|_p^p+\| F(\nabla u)\|_{q,\mu}^q +\|u\|_p^p+\|u\|_{q,\mu}^q-\|u\|_{p^*}^{p^*}-\lambda\int_\Omega u^{\gamma}\diff x
		\nonumber\\
		&\quad -\lambda\int_\Omega g_1(x,u)u\diff x-\|u\|_{p_*,\partial\Omega}^{p_*}-\int_{\partial \Omega}g_2\l(x,u\r)u\diff \sigma
		\nonumber\\
		&\quad +\varepsilon \int_\Omega F^{p-1}(\nabla u)\nabla F(\nabla u)\cdot\nabla v \diff x+\varepsilon \int_\Omega\mu(x) F^{q-1}(\nabla u)\nabla F(\nabla u)\cdot\nabla v \diff x
		\nonumber\\
		&\quad +\varepsilon \int_\Omega u^{p-1} v \diff x+\varepsilon \int_\Omega\mu(x) u^{q-1}v \diff x-\eps\int_\Omega u^{p^*-1}v\diff x- \eps \lambda\int_\Omega u^{\gamma-1}v \diff x
		\nonumber\\
		&\quad -\eps \lambda\int_\Omega g_1(x,u)v\diff x-\varepsilon\int_{\partial \Omega}u^{p_*-1}v\diff \sigma-\eps \int_{\partial\Omega} g_2(x,u)v\diff \sigma
		\nonumber\\
		&\quad  -\int_{\{u+\eps v< 0\}}F^p(\nabla u)\diff x-\varepsilon\int_{\{u+\eps v< 0\}}F^{p-1}(\nabla u)\nabla F(\nabla u)\cdot\nabla v \diff x\\
		&\quad -\int_{\{u+\eps v< 0\}}\mu(x)F^q(\nabla u)\diff x-\varepsilon\int_{\{u+\eps v< 0\}}\mu(x) F^{q-1}(\nabla u)\nabla F(\nabla u)\cdot\nabla v \diff x
		\nonumber\\
		&\quad  -\int_{\{u+\eps v< 0\}} u^{p-1}(u+\eps v)\diff x-\int_{\{u+\eps v< 0\}} \mu(x)u^{q-1}(u+\eps v)\diff x
		\nonumber\\
		&\quad +\int_{\{u+\eps v< 0\}} u^{p^*-1}(u+\eps v)\diff x+\lambda\int_{\{u+\eps v<0\}} u^{\gamma-1}(u+\eps v)\diff x
		\nonumber\\
		&\quad +\lambda\int_{\{u+\eps v< 0\}} g_1(x,u)(u+\eps v)\diff x+\int_{\partial \Omega}u^{p_*-1}(u+\eps v)\diff \sigma
		\nonumber\\
		&\quad +\int_{\partial \Omega}g_2(x,u)(u+\eps v)\diff \sigma
		\nonumber\\
		&\leq \varepsilon\left[ \int_\Omega \l(F^{p-1}(\nabla u)+\mu(x) F^{q-1}(\nabla u)\r)\nabla F(\nabla u)\cdot\nabla v \diff x+\int_\Omega u^{p-1} v \diff x\r.
		\nonumber\\
		& \l. \qquad+ \int_\Omega\mu(x) u^{q-1}v \diff x-\int_\Omega u^{p^*-1}v\diff x-\lambda\int_\Omega u^{\gamma-1}v\diff x-\lambda\int_\Omega g_1(x,u) v\diff  x\right.
		\nonumber\\
		&\left. \qquad -\int_{\partial \Omega}u^{p_*-1}v\diff \sigma- \int_{\partial\Omega} g_2(x,u)v\diff \sigma \right]
		\nonumber\\
		&\quad -\varepsilon\int_{\{u+\eps v< 0\}}F^{p-1}(\nabla u)\nabla F(\nabla u)\cdot \nabla v\diff x
		\nonumber\\
		& \quad -\varepsilon\int_{\{u+\eps v< 0\}}\mu(x) F^{q-1}(\nabla u)\nabla F(\nabla u)\cdot\nabla v \diff x
		\nonumber\\
		&\quad  -\eps \int_{\{u+\eps v< 0\}} u^{p-1}v\diff x-\eps \int_{\{u+\eps v< 0\}} \mu(x)u^{q-1}v\diff x.\nonumber
	\end{align}
	Note that the measure of the set $\{u+\varepsilon v< 0\}$ goes to $0$ as $\varepsilon\to 0$. Hence,
	\begin{align*}
		&\int_{\{u+\varepsilon v < 0\}}F^{p-1}(\nabla u)\nabla F(\nabla u)\nabla v \diff x \to 0 \quad\text{as } \eps \to 0,\\
		&\int_{\{u+\varepsilon v< 0\}}\mu(x) F^{q-1}(\nabla u)\nabla F(\nabla u)\nabla v\diff x \to 0 \quad\text{as } \eps \to 0,\\
		& \int_{\{u+\eps v< 0\}} u^{p-1}v\diff x \to 0 \quad\text{as } \eps \to 0,\\
		&  \int_{\{u+\eps v< 0\}} \mu(x)u^{q-1}v\diff x\to 0 \quad\text{as } \eps \to 0.
	\end{align*}
	Therefore, dividing the inequality \eqref{last} by $\eps$ and passing to the limit as $\varepsilon\to 0$ we conclude that
	\begin{align*}
			& \int_{\Omega}\left(F(\nabla u)^{p-1}+\mu(x)F(\nabla u)^{q-1}\right)\nabla F(\nabla u)\cdot\nabla v\diff x + \int_{\Omega} u^{p-1}v \diff x\\ 
			& +\int_{\Omega} \mu(x) u^{q-1}v\diff x-\int_{\Omega}u^{p^{*}-1} v\diff x -\lambda\int_{\Omega}u^{\gamma-1} v\diff x-\lambda\int_{\Omega}g_1(x,u) v\diff x\\
			&-\int_{\partial\Omega} u^{p_*-1}v \diff \sigma - \int_{\partial\Omega} g_2(x,u)v \diff \sigma\geq 0. 
	\end{align*}
	Since $v\in\W$ was arbitrary chosen, we see from the last inequality that equality must hold. Therefore, $u\in\W$ is a weak solution of problem \eqref{problem} in the sense of Definition \ref{def_weak_solution}.
\end{proof}

\section*{Acknowledgments}
The authors wish to thank the knowledgeable referee for his/her remarks in order to improve the paper.\\
C.\,Farkas was supported by the National Research, Development and Innovation Fund of Hungary, financed under the K\_18 funding scheme, Project No.\,127926 and by the Sapientia Foundation - Institute for Scientific Research, Romania, Project No.\,17/11.06.2019.

A.\,Fiscella is member of the {Gruppo Nazionale per l'Analisi Ma\-tema\-tica, la Probabilit\`a e
le loro Applicazioni} (GNAMPA) of the {Istituto Nazionale di Alta Matematica ``G. Severi"} (INdAM).
A.\,Fiscella realized the manuscript within the auspices of the INdAM-GNAMPA project titled "Equazioni alle derivate parziali: problemi e modelli" (Prot\_20191219-143223-545), of the FAPESP Project titled "Operators with non standard growth" (2019/23917-3), of the FAPESP Thematic Project titled "Systems and partial differential equations" (2019/02512-5) and of the CNPq Project titled "Variational methods for singular fractional problems" (3787749185990982).

\end{document}